\documentclass[12pt]{amsart}
\usepackage{amsfonts}
\usepackage{amsmath}
\usepackage{amsxtra}
\usepackage{amssymb,latexsym}
\usepackage[mathcal]{eucal}
\usepackage{amscd}

\makeindex

\newtheorem{theo}{{\bfseries Theorem}}[section]
\newtheorem{prop}[theo]{{\bfseries Proposition}}
\newtheorem{lem}[theo]{{\bfseries Lemma}}

\newtheorem{df}[theo]{{\bfseries Definition}}

\def \N {\mathbb N}

\def \Z {\mathbb Z}

\def \A {\mathcal A}
\def \B {\mathcal B}
\def \CC {\mathcal C}

\def \D {\mathcal D}

\def \O {\mathcal O}

\def \r {\rho}
\def \s {\sigma}

\def \tto {\longrightarrow}

\usepackage{amssymb,latexsym}
\usepackage[mathcal]{eucal}
\usepackage{amscd}
\usepackage{amsmath}
\usepackage{graphicx}
\usepackage{amsfonts}
\usepackage{amscd}

\numberwithin{equation}{section}

\begin{document}

\title[Partition Constructions]{\bfseries  Generalized Intransitive Dice II: \\ Partition Constructions}
\vspace{1cm}
\author{Ethan Akin and Julia Saccamano}
\address{Mathematics Department \\
    The City College \\ 137 Street and Convent Avenue \\
       New York City, NY 10031, USA     }
\email{ethanakin@earthlink.net, julia.saccamano@macaulay.cuny.edu}

\date{May, 2019}

\begin{abstract} A generalized $N$-sided die is a random variable $D$ on a sample space of $N$ equally likely outcomes taking values in
the set of positive integers. We say of independent $N$ sided dice $D_i, D_j$ that $D_i$ beats $D_j$, written
$D_i \to D_j$, if $Prob(D_i > D_j) > \frac{1}{2} $. A collection of dice $\{ D_i : i = 1, \dots, n \}$ models a tournament on the set
$[n] =  \{ 1, 2, \dots, n \}$, i.e. a complete digraph with $n$ vertices, when $D_i \to D_j$ if and only if $i \to j$ in the tournament.
By using $n$-fold partitions of the set $[Nn] $ with each set of size $N$ we can  model an arbitrary tournament on $[n]$.
A bound on the required size of $N$ is obtained by  examples with $N = 3^{n-2}$.
.
\end{abstract}


\vspace{1cm}

\vspace{.5cm} \maketitle

\tableofcontents

\section{Introduction: Tournaments Using Dice}\label{sec1}
\vspace{.5cm}

A generalized die is a cube with each face labeled with a positive number. The possibility of repeated labels is allowed.
On the standard die each of numbers $1,2, \dots 6$ occurs once. We identify the die with the random variable of the
outcome of a roll with each face equally likely. With a pair of dice we assume that the rolls are independent.

Of two dice $D_1$ and $D_2$
we say that $D_1$ beats $D_2$ (written $D_1 \to D_2$) if the probability that
$D_1 > D_2$ is greater than $\frac{1}{2}$
where  $D_1$ and $D_2$ are the independent outcomes of the rolls of the dice.

Of course, with standard dice this does not happen.  If $D_1$ and $D_2$ are standard, then $P(D_1 > D_2) = 15/36$.

If we use the labels
\begin{align}\label{exeq00}
 \begin{split}
A_1 \ &=  \ \{  3 \}, \\
A_2 \ &=  \ \{ 2 \}, \\
A_3 \ &=  \ \{  1\}.
 \end{split}
 \end{align}
 and repeat each label six times to get $6$-sided dice, then clearly always $D_1 \to D_2, D_3$ and $D_2 \to D_3$.

However, there exist  examples of nontransitive dice, or \emph{intransitive dice}, \index{intransitive dice}
three dice $D_1, D_2, D_3$ such that $D_1 \to D_2, D_2 \to D_3,$ and $D_3 \to D_1$. For example, if we let
\begin{align}\label{exeq01}
 \begin{split}
A_1 \ &=  \ \{  3, 5, 7 \}, \\
A_2 \ &=  \ \{ 2, 4, 9 \}, \\
A_3 \ &=  \ \{  1, 6, 8 \}.
 \end{split}
 \end{align}
and repeat each label twice to get $6$-sided dice, $\{ D_1, D_2, D_3 \}$, then $P(D_{i} > D_{i+1}) = \frac{5}{9}$ for $i=1,2,3$ (counting mod $3$).

A \emph{digraph}\index{digraph} $R$ on a set $I$ of size (= cardinality) $|I|$\index{$|I|$} is a set of ordered pairs $(i,j)$  of
distinct elements $i, j \in I$ such that at most one of the pairs $(i,j), (j,i)$ lies in $R$.
The digraph is called a \emph{tournament}\index{tournament} when exactly one of the pairs $(i,j), (j,i)$ lies in $R$.
The name arises because $R$ models the outcomes of a round-robin tournament where every
pair of players competes once with $i$ beating $j$, written $i \to j$, if $(i,j) \in R$. Alternatively, we can
think of $I$ as a list of strategies or actions
so that $i \to j$ when $i$ wins against $j$. The output set $R(i) = \{ j : i \to j \}$ consists of
the elements of $I$ which are beaten by $i$.

Up to relabeling, there are two tournaments of size $3$: the ordering $\{ (1,2),(2,3),(1,3) \}$ and
the $3$-cycle $\{ (1,2),(2,3),(3,1) \}$. The above
examples show that each can be mimicked by using dice.

The $3$-cycle models the game Rock-Paper-Scissors. In general, we will call a tournament $R$ on $I$ a \emph{game} when its size $|I|$ is odd and for
each $i$, $|R(i)| = \frac{1}{2}(|I| - 1)$.  That is, each strategy beats exactly half of its competing strategies and is beaten by the other half.
Clearly, the $3$-cycle is (up to isomorphism) the only game of size $3$.

With size $5$ there is also a game which is unique up to isomorphism. On the television show \emph{The Big Bang Theory} this game was described
as Rock-Paper-Scissors-Lizard-Spock.  It can also be modeled using generalized dice:
 \begin{align}\label{exeq03}
 \begin{split}
 A \ &= \ \{  1, 6, 10, 22, 24, 30 \}, \\
 B \ &= \ \{ 7, 12, 13, 15, 19, 27 \}, \\
 C \ &= \ \{  3, 4, 17, 18, 23, 28 \}, \\
 D \ &= \ \{  2, 9, 11, 16, 26, 29 \}, \\
 E \ &= \ \{  5, 8, 14, 20, 21, 25 \},
 \end{split}
 \end{align}
which satisfy, with probability $19/36$ in each case,
\begin{equation}\label{exeq04}
  A \to C, E; \ B \to A, D; \ C \to B, D; D \to A, E; E \to B, C.
  \end{equation}

We would like to similarly mimic an arbitrary tournament.  However, as the size of the tournament grows we will require larger dice, dice with
more than $6$ ``faces''.
On a sample space of $N$ equally likely outcomes, which we will call the \emph{faces}, an \emph{$N$ sided die}\index{$N$ sided die}
 is a random variable
taking positive integer values. Again a pair of competing $N$-sided dice are assumed independent. See, for example, \cite{CG}.

An $N$ sided die is called \emph{proper} \index{$N$ sided die!proper} when the values are all in $[N] = \{1, \dots, N \}$ and the
sum of the values is $\frac{(N + 1)N}{2}$, or, equivalently, the expected value is $\frac{N + 1}{2}$ which is the same as that of
the \emph{standard} $N$ sided die \index{$N$ sided die!standard} which takes on each value of $[N]$ once.

In \cite{A2} it is shown in various ways that an arbitrary tournament can be modeled by proper $N$ sided dice.

A convenient way of constructing examples is by the use of partitions.

A \emph{partition}\index{partition} $\A $ of a set $I$ is a collection of  disjoint subsets with union $I$. We call it a
\emph{regular  partition} \index{partition!regular} when the cardinalities of the elements of $\A$ are all the same.
 From now on we will assume that our partitions are regular so that
$\A = \{ A_1, \dots A_n \}$ is an $n$ partition of $[Nn] = \{ 1, \dots, Nn \}$ when it is a partition of $[Nn]$ with  $|A_i| = N$ for $i = 1, \dots n$.
There are $(Nn)!/(N!)^n$  $n$ partitions of $[Nn]$.

For a finite subset $A \subset \N$ we denote by $\s(A)$\index{$\s(A)$} the sum of the elements of $A$. If $\A = \{ A_1, \dots A_n \}$ is an
$n$ partition of $[Nn]$, we call the partition \emph{proper}\index{partition!proper} when the sums $\s(A_i)$ are all equal and so
\begin{equation}\label{exeq04a}
\s(A_i) \ = \ \frac{N(Nn + 1)}{2} \quad \text{for} \ i \in [n].
\end{equation}
Equivalently, for each $i$, the expected value of a random element of $A_i$ is $\frac{Nn + 1}{2}$.

For an $n$ partition $\A$ on $[Nn]$ we define the digraph
\begin{equation}\label{exeq05}
R[\A] \ = \ \{ (i,j) \in [n] \times [n]: |\{ (a,b) \in A_i \times A_j : a > b \}| > N^2/2 \}.
\end{equation}
That is, $(i,j) \in R[\A]$ or $A_i \to A_j$ if it is more likely that a randomly chosen element of $A_i$ is greater
than a randomly chosen element of $A_j$ rather
than the reverse.

If $N$ is  odd, then $R[\A]$ is necessarily a tournament on $[n]$.  That is, for every pair $i, j \in [n]$ with $i \not= j$ either
$A_i \to A_j$ or $A_j \to A_i$. Note that for $i = j$, $|\{ (a,b) \in A_i \times A_i : a > b \}| = N(N-1)/2$.

We can use the partition to label the faces of $n$ different $N$ sided dice,   with distinct values selected from $[Nn]$.
If $D_i$ is the random variable associated with the die labeled with values from $A_i$, then
$A_i \to A_j$ exactly when $D_i \to D_j$ in the previous sense. If $\A$ is a proper $n$ partition of $[Nn]$
then repeating each value $n$ times, we obtain $n$ proper $Nn$ sided dice $\{ D_i : i \in [n] \}$ with $D_i \to D_j$ if and only if
$A_i \to A_j$.

Example (\ref{exeq01})
is a proper $3$ partition of $[9]$, and Example (\ref{exeq03}) is a proper $5$ partition of $30$.

We will say that an $n$ partition $\A$ of $[Nn]$ \emph{models} a tournament \index{tournament!model} $R$ on $[n]$ when $R[\A] = R$.

For $N$ large enough we can model any tournament on $[n]$ by using an $n$ partition on $[Nn]$. In \cite{A2} the following is proved.

\begin{theo}\label{maintheopart} If $R$ is a tournament on $[n]$, then there is a positive integer $M$ such that for every integer $N \geq M$,
there exists an $n$ partition of $[Nn]$
$\A = \{ A_1, \dots, A_n \}$
such that for $i, j \in [n]$, $A_i \to A_j$ if and only if $i \to j$ in $R$.  That is, $R = R[\A]$.
\end{theo}
\vspace{.5cm}

However the proof of this theorem and the related results in \cite{A2} are all rather non-constructive.  If we let $N_n$ be the smallest
positive integer $N$ such that every tournament on $[n]$ can be modeled by an $n$ partition on $[Nn]$, then the results of  \cite{A2}
do not provide a bound
on the size $N_n$. In Section \ref{sec3} we provide an explicit construction
which will yield such a bound.  In Theorem \ref{partcor06} below we will show the following

\begin{theo}\label{maintheopart2} If $R$ is a tournament on $[n]$ with $n \geq 2$, then
there exists an $n$ partition of $[3^{n-2}n]$
$\A = \{ A_1, \dots, A_n \}$
such that  $R = R[\A]$. In particular, $N_n \leq 3^{n-2}$.
\end{theo}
\vspace{.5cm}

The bound is probably very crude. Furthermore, the examples constructed are not necessarily proper. On the other hand,
 we will show that for arbitrary positive $n$, there is a game of size $2n+1$ which can be
modeled by a proper $2n+1$ partition of $[3(2n+1)]$.

Notice that  $N_n$ does tend to infinity with $n$. To see this, recall that the number of $n$ partitions of $[Nn]$
is $P_n = (Nn)!/(N!)^n$. Since $\int_{1}^{n+1} \ \ln(x) \ dx > \ln( n!) >  \int_{1}^{n} \ \ln(x) \ dx$,
\begin{align}\label{exeq06}
    \begin{split}
    \ln( P_n) \ \leq \ [(Nn+1)&(\ln(Nn + 1) - 1) + 1] - n[N(\ln(N) - 1) + 1]  \\
     = \ (Nn + 1)&(\ln(n + (1/N)) + \ln(N) - n.
 \end{split}
    \end{align}
On the other hand, the number of tournaments of size $n$ is $T_n = 2^{n(n-1)/2}$ and so $\ln( T_n) = \frac{\ln(2)}{2} n(n-1)$. Because $n^2$ grows faster
than $n \ln(n)$ it follows that $N_n$ cannot remain bounded as $n $ tends to infinity.
\vspace{1cm}

\section{Tournaments and Games}\label{sec2}
\vspace{.5cm}

All the sets we consider are  assumed to be finite.

A \emph{digraph}\index{digraph} on a  nonempty set $I$ is a subset $R \subset I \times I$ such that
$R \cap R^{-1} = \emptyset$ with $R^{-1} = \{ (j,i) : (i,j) \in R \}$.
In particular, $R$ is disjoint from
the diagonal $\Delta =\{(i,i) : i \in I\}$.
 We write $i \to j$ for $(i,j) \in R$.
For a vertex $i$, i.e. $i \in I,$ the \emph{output set}\index{output set} $ R(i) = \{ j : (i,j) \in R \}$ and so $R^{-1}(i) = \{ j : (j,i) \in R \}$
is the \emph{input set}\index{input set}.
If $J \subset I$, then the \emph{restriction}\index{restriction} of $R$ to $J$ is $R|J = R \cap (J\times J)$.

We use $|I|$ to denote the cardinality of a set $I$. Notice that if $I $ is the singleton $ \{ u \}$, then $R = \emptyset$ is the only
digraph on $I$.  We call this the \emph{trivial} digraph and denote it $\emptyset[u]$

Given a map $\r : I \tto J$ we let $\bar \r$\index{$\bar \r$}  denote the product map $\r \times \r : I \times I \tto J \times J$.

\begin{df}\label{gamedef01} Let $R$ and $S$ be digraphs on $I$ and $J$, respectively. A
\emph{morphism} $\r : R \tto S$ is a
map $\r : I \tto J$ such that $(\bar \r)^{-1}(S) = R \setminus (\bar \r)^{-1}(\Delta_J)$.  That is, for $i_1, i_2 \in I$ with
$\r(i_1) \not= \r(i_2)$ $\r(i_1) \to \r(i_2)$ if and only if $i_1 \to i_2$. 
\end{df}

Clearly, if $\r$ is a bijective morphism then $\r^{-1}$ is a morphism and so $\r$ is an \emph{isomorphism}.
Two digraphs are isomorphic when each can be obtained from the
other by relabeling the vertices.

An \emph{automorphism} of $R$ is an
isomorphism with $R = S$.

If $R$ is a digraph on $I$ and $\pi$ is a permutation of $I$, then we let $\pi R$ be the digraph on $I$ given by
\begin{equation}\label{gameeq00aaa}
\pi R \ = \ \bar \pi (R) \ = \ \{ (\pi(i),\pi(j)) : (i,j) \in R \}.
\end{equation}
Clearly, if $R$ and $S$ are digraphs on $I$, then
 $\r : R \tto S$ is an isomorphism if and only if the map $\r$ on $I$ is bijective, i.e. a permutation, and $S = \r R$.

 An $R$ \emph{path}\index{path} $[i_0,\dots,i_n] $ from $i_0$ to $i_n$ (or simply a path when $R$ is understood)  is a
 sequence of elements of $I$ with $n \geq 1$ such that
$(i_k,i_{k+1}) \in R$ for $k = 0,\dots, n-1$. The length of the path is $n$. It is a \emph{closed path} when $i_n = i_0$.
An\emph{ $n$ cycle}\index{cycle}, denoted $\langle i_1, \dots, i_n \rangle$, is
a closed path $[i_n, i_1, \dots, i_n]$ such that the vertices $i_1, \dots, i_n$ are distinct.
A path \emph{spans} $I$ when every $i \in I$ occurs on the path. A spanning cycle is called a \emph{Hamiltonian cycle}\index{cycle!Hamiltonian} for $R$.

A digraph $R$ is called strongly connected, or just \emph{strong}\index{digraph!strong}, if for every pair $i,j$ of distinct elements of $I$
there is a path from $i$ to $j$. It follows that
if $|I| > 1$, then for any $i \in I$ there is a path beginning and ending at $i$. We may eliminate any repeated vertices $j_k = j_{\ell}$ with
$0 < k < \ell$ by removing the portion of the path $j_k, j_{k+1}, \dots, j_{\ell-1}$ and renumbering.  This shows that if $R$ is strong and nontrivial,
there is a cycle through each vertex. The trivial digraph on a singleton is strong vacuously.

A subset $J \subset I$ is \emph{invariant}\index{invariant set} if $i \in J$  implies that the output set $R(i)$ is contained in $J$,
or, equivalently, if any path which begins in $J$ remains in $J$.
It is clear that $R$ is strong if and only if it does not contain any proper invariant subset.

A digraph $R$ is called a \emph{tournament}\index{tournament} when  $R \cup R^{-1} = (I\times I) \setminus \Delta$.
Thus, $R$ is a tournament on $I$ when for each pair of distinct elements $i, j \in I$
either $(i,j)$ or $(j,i)$ lies in $R$ but not both. Clearly, if $R$ is a tournament on $I$ and $J \subset I$, then the
restriction $R|J$ is a tournament on $J$. Harary and Moser provide a nice exposition of tournaments in \cite{HM}.


%
%

\begin{prop}\label{gameprop01} If $R$ is a strong tournament on $I$ with $|I| = p > 1$ and $i \in I$,
then for every $\ell$ with $3 \leq \ell \leq p$
there exists a $\ell$-cycle  in $R$ passing through $i$. In particular, $R$ is a strong, nontrivial tournament if and only if
it admits a Hamiltonian cycle.
\end{prop}

\begin{proof} See Moon, \cite{M} Theorem 3 for a proof of this result which is a sharpening of Harary and Moser \cite{HM} Theorem 7. See also
\cite{A1} Proposition 1.2.

\end{proof}\vspace{.5cm}

For $S'$ and $S$  tournaments on $J', J$, respectively, with $J'$ and $J$ disjoint,
the \emph{domination product}\index{domination product} is  the tournament $S' \rhd S$ on $J' \cup J$ defined by:
  \begin{equation}\label{gameeq01}
  S' \rhd S \ = \ S' \cup (J' \times J) \cup S,
  \end{equation}
  so that $J$ is a proper invariant subset for $S' \rhd S$.

  Conversely, if $J$ is a proper invariant subset for a tournament $R$ on $I$ and $J' = I \setminus J$, then
    \begin{equation}\label{gameeq02}
 R \ = \ (R|J') \rhd (R|J).
  \end{equation}

   Let $R$ be a nontrivial tournament on $I$, $v \in I$ and $J = I \setminus \{ v \}$.  The vertex $v$ is called a \emph{maximum}
   when it satisfies the following equivalent conditions
   \begin{itemize}
   \item $v \to u$ for all $u \not= v$ in $I$.
   \item $R(v) = J$
   \item $R^{-1}(v) = \emptyset$.
   \item  $R = (\emptyset[v]) \rhd (R|J)$.
   \end{itemize}
   Similarly, vertex $v$ is a \emph{minimum}
   when it satisfies the following equivalent conditions
   \begin{itemize}
   \item $u \to v$ for all $u \not= v$ in $I$.
   \item $R^{-1}(v) = J$
    \item $R(v) = \emptyset$.
   \item  $R = (R|J) \rhd (\emptyset[v])$.
   \end{itemize}
\vspace{.5cm}

  \begin{prop}\label{gameprop02} Let $R$ be a nontrivial tournament on $I$.

  (a) The tournament $R$ is not strong if and only if it is a domination product, i.e. $ R =  S' \rhd S$ for some tournaments $S$ and $S'$.

  (b) If $v \in I$ with $J = I \setminus \{ v \}$ and $R|J$ is strong, then exactly one of the following is true.
\begin{itemize}
\item[(i)]  The vertex $v$ is a maximum vertex for $R$.
\item[(ii)]  The vertex $v$ is a minimum vertex for $R$.
\item[(iii)] The tournament $R$ is strong.
\end{itemize}
   \end{prop}

   \begin{proof} (a) This follows from (\ref{gameeq02}) and the remarks before it.

   (b)    Because $R|J$ is strong, there is a Hamiltonian cycle $\langle i_0, \dots, i_p \rangle$ for $R|J$. If $R^{-1}(v) \not= \emptyset$ we can renumber and
   so assume $i_0 \in R^{-1}(v)$. Let $k$ be the maximum integer such that $i_q \in R^{-1}(v)$ for $0 \leq q \leq k$. If also $R(v) \not= \emptyset$
   then $k < p$ and so $i_{k+1} \in R(v)$. Thus, $\langle i_0, \dots, i_k, v, i_{k+1}, \dots, i_p \rangle$ is a Hamiltonian cycle for $R$ and so $R$ is strong.

\end{proof}\vspace{.5cm}

For a positive  integer $k$, a digraph $R$ is called $k$ \emph{regular}\index{digraph!regular} when  both the input set and the output set of
of every vertex have cardinality $k$.  That is, $|R(i)| = |R^{-1}(i)| = k$ for all $i \in I$. A digraph which is $k$ regular for
some $k$ is called \emph{regular}. If a tournament on $I$ is $k$ regular, then $|I| = 2k + 1$. We will call a regular tournament
 a \emph{game}\index{game} because such a tournament generalizes the Rock-Paper-Scissors game.  Such games are described in \cite{A1}. In particular,
 it is demonstrated there that up to isomorphism there is a unique game of size $5$.

  Of special interest are the \emph{group games}\index{group game} described in Section 3 of \cite{A1}.

  Let $\Z_{2n+1}$ denote the additive group of integers mod $2n+1$ with congruence classes labeled by $0, 1, \dots, 2n$.
  Call $A \subset \Z_{2n+1}$ a \emph{game subset}\index{game subset} if $A \cap -A = \emptyset $ and $\Z_{2n+1} = \{ 0 \} \cup A \cup -A$ where
  $-A = \{ -a : a \in A \}$. In particular, $|A| = n$. The set $\Z_{2n+1} \setminus \{ 0 \}$ is decomposed by the $n$ pairs
  $\{ \{a, -a \} : a \in  \Z_{2n+1} \setminus \{ 0 \} \}$ and a game subset is obtained by choosing one element from each pair.
  In particular, there are $2^n$ game subsets. For example $[n] = \{ 1, 2, \dots, n \}$ is a game subset.

  For any game subset $A$ define the associated game $R[A]$ on $\Z_{2n+1}$ by
  \begin{equation}\label{gameeq03}
  R[A] \ = \ \{ (i,j) : j - i \in A \}.
  \end{equation}
  Since $R[A](i) = i + A, R[A]^{-1}(i) = i - A$ it follows that $R[A]$ is a regular tournament, i.e. a game.
  Furthermore, the translation map $\ell_j$ on $\Z_{2n+1}$ defined by $\ell_j(i) = j + i$ is an automorphism of $R[A]$ for each
  $j \in \Z_{2n+1}$.

  It follows that for every positive integer $n$ there is a game of size $2n+1$. Another way of seeing this is by induction using the
  following construction.

  Let $R$ be a tournament on $I$. With $J \subset I$, let $J' = I \setminus J$. For $u, v$ distinct vertices not in $I$, define the tournament
  $R^+$, called the \emph{extension} of $R$ via $J$ and $u \to v$, by
    \begin{align}\label{gameeq04}
    \begin{split}
  R^+|I \ &= \ R,  \\
  R^+(u) \ &= \  \{ v \} \cup J', \\
  R^+(v) \ &= \ J,
  \end{split}
  \end{align}
  so that $  (R^+)^{-1}(u) = J$ and  $(R^+)^{-1}(v) = \{ u \} \cup J'$.

  If $R$ is a game with $|I| = 2n - 1$ and $|J| = n$, then the extension $R^+$ is a game of size $2n + 1$.

  We conclude the section with the definition of the lexicographic product following the definition in
  \cite{S1} and \cite{S2} for graphs and in \cite{GM} for tournaments, see also \cite{A1} Section 6.

For $R, S$ digraphs on $I, J$, respectively,  the \emph{lexicographic product}\index{lexicographic product}
 $R \ltimes S$ \index{$R \ltimes S$} is a digraph on $I \times J$. For $u, v \in I \times J$ we define $u \to v$ when
 \begin{equation}\label{gameeq06}
u_1 \to v_1 \ \text{in} \ R, \qquad \text{or} \qquad u_1 = v_1 \ \text{and} \ u_2 \to v_2 \ \text{in} \ S.
\end{equation}

It is easy to check that $R \ltimes S$ is a tournament (or a game) if both $R$ and $S$ are tournaments (resp. both are games).

\vspace{1cm}

 \section{Partition Constructions}\label{sec3}
\vspace{.5cm}

Recall that we defined for an $n$ partition $\A$ of $[Nn]$ the digraph
\begin{equation}\label{exeq05aaa}
R[\A] \ = \ \{ (i,j) \in [n] \times [n]: |\{ (a,b) \in A_i \times A_j : a > b \}| > N^2/2 \}.
\end{equation}
If $R[\A]$ is a tournament, e.g. if $N$ is odd, then by permuting $[n]$ or, equivalently, by relabelling the elements of $\A$,
we can obtain every tournament isomorphic to $R[\A]$ as the tournament of an $n$ partition of $[Nn]$.

For two disjoint sets $A, B \subset \N$  we define
\begin{equation}\label{parteq11}
Q(A,B) \ = \ 2 \cdot |\{ (a,b) \in A \times B : a > b \}| - |A| \cdot |B|.
\end{equation}
 \index{$Q(A,B)$} Clearly,
 \begin{equation}\label{parteq11ccc}
 Q(A,B) \ \equiv \ |A| \cdot |B| \quad \text{mod} \ 2.
 \end{equation}

We will write $A \to B$ when $Q(A,B) > 0$. Observe that
\begin{equation}\label{parteq11aaa}
\frac{Q(A,B)}{2 |A| \cdot |B|} \ = \ P( a > b) - \frac{1}{2}
\end{equation}
where $a$ and $b$ are chosen randomly from $A$ and $B$, respectively.  In particular, if
$\A = \{ A_1, \dots, A_n \}$ is an $n$  partition of $[Nn]$, then for the  digraph $R[\A]$ on $[n]$
\begin{equation}\label{parteq16}
\begin{split}
i \to j \quad \Longleftrightarrow \quad Q(A_i,A_j) > 0, \\
\end{split}
\end{equation}

 \vspace{.5cm}

\begin{lem} \label{partlem00} For disjoint  sets $A, B \subset \N$
\begin{equation}\label{parteq11a}
Q(A,B) \ = \ |\{ (a,b) \in A \times B : a > b \}| - |\{ (a,b) \in A \times B : b > a \}|,
\end{equation}
\begin{equation}\label{parteq12}
Q(B,A) \ = \ - \ Q(A,B).
\end{equation}
\end{lem}

\begin{proof} Clearly, $|\{ (a,b) \in A \times B : a > b \}| + |\{ (a,b) \in A \times B : b > a \}| = |A| \cdot |B|$ since
$A$ and $B$ are disjoint. From this and (\ref{parteq11}) (\ref{parteq12}) is obvious.
Subtract from the equation $2 |\{ (a,b) \in A \times B : a > b \}| = |A| \cdot |B| + Q(A,B)$,
to get (\ref{parteq11a}).

\end{proof}
 \vspace{.5cm}

If $|A|$ and $|B|$ are odd, then by (\ref{parteq11ccc}) $Q(A,B)$ is odd and so cannot equal zero.

On the other hand, if
$A = \{ a_1 < a_2 \}, B = \{ b_1 < b_2 \} \subset \N$,  then
\begin{align}\label{parteq14}
\begin{split}
\text{Case} (i) \ a_1 < b_1 < b_2 < a_2
 \quad \Longrightarrow &\quad Q(A,B) = 2 \cdot 2 - 4 = 0, \\
\text{Case} (ii) \ b_1 < a_1 < b_2 < a_2 \quad \Longrightarrow &\quad Q(A,B) = 2 \cdot 3 - 4 = 2.
\end{split}
\end{align}
We call Case (i) a \emph{pair inclusion}\index{pair inclusion} which we will write as $B \hookrightarrow A$ \index{$B \hookrightarrow A$}
and Case (ii) a \emph{pair overlap}\index{pair overlap} with $A$ higher which
we will write as $A \twoheadrightarrow B$.\index{$A \twoheadrightarrow B$}

We write $B < A$ when $b < a$ for all $(a,b) \in A \times B$.  In that case, $Q(A,B) \ = |A| \cdot |B|$.
 \vspace{.5cm}

\begin{lem}\label{partlem01} Let $A_1, \dots, A_{k}, B_1, \dots, B_{\ell}$ be pairwise disjoint finite
subsets of $\N$.  With $A = \bigcup_{i=1}^k \  A_i, B = \bigcup_{j=1}^{\ell} \  B_j$
\begin{equation}\label{parteq13}
Q(A,B) \ = \ \sum_{(i,j) \in [k] \times [\ell]} \ Q(A_i,B_j).
\end{equation}

If $k = \ell$, $|A_i| \cdot |B_j| = |A_j| \cdot |B_i|$ for $i,j = 1, \dots, k$ and $A_1 \cup B_1 < A_2 \cup B_2, \dots < A_k \cup B_k$, then
\begin{equation}\label{parteq13a}
Q(A,B) \ = \ \sum_{i \in [k]} \ Q(A_i,B_i).
\end{equation}
\end{lem}

\begin{proof} Every $(a,b) \in A \times B$ is in a unique $A_i \times B_j$. Furthermore,
$|A| \cdot |B| = \sum_{(i,j) \in [k] \times [\ell]} \ |A_i| \cdot |B_j$. These imply (\ref{parteq13}).

If $A_i \cup B_i < A_j \cup B_j$ with $|A_i| \cdot |B_j| =|A_j|\cdot |B_i|$ then
$Q(A_j,B_i) = - Q(A_i,B_j) = |A_j| \cdot |B_i|.$  So (\ref{parteq13a}) follows from (\ref{parteq13}).

\end{proof}
 \vspace{.5cm}

For a triple $A = \{a_1 < a_2 < a_3 \} $  we write $A^- = \{ a_1, a_2 \}$ and $A^+ = \{ a_2, a_3 \}$.

\begin{lem}\label{triplelem} Assume $A = \{a_1 < a_2 < a_3 \} $ and $B = \{b_1 < b_2 < b_3 \} $ are disjoint subsets of $\N$.

(a) If $A^-$ and $B^-$ are inclusion pairs, then $a_3 > b_3 > a_2$ implies $Q(A,B) = 1$. If $A^+$ and $B^+$ are inclusion pairs,
then $b_2 > a_1 > b_1$ implies $Q(A,B) = 1$.

(b) If $A^- \twoheadrightarrow B^-$, then  $a_3 > b_3 > a_2$ implies $Q(A,B) = 3$ and $b_3 > a_3 > b_2$ implies $Q(A,B) = 1$.
If $A^+ \twoheadrightarrow B^+$, then  $b_2 > a_1 > b_1$ implies $Q(A,B) = 3$ and $a_2 > b_1 > a_1$ implies $Q(A,B) = 1$. In particular, in all of
these cases $Q(A,B) > 0$.
\end{lem}

\begin{proof}  These are obvious from (\ref{parteq13a}) and (\ref{parteq14}) or direct computation.

\end{proof}
\vspace{.5cm}

 Let $\A = \{ A_1, \dots, A_n \}$ be an $n$ partition of $[Nn]$,
 and $M$ be a positive integer.
Define $\A^* = \{ A^*_1, \dots, A^*_n \}$ and
$  \A^{(M)} = \{ A^{(M)}_1, \dots, A^{(M)}_n \}$,
$n$ partitions of $[Nn]$ and of $[MNn]$, respectively, by
\begin{align}\label{inteq17part}
\begin{split}
A^*_i \ &= \ \{ Nn - a + 1 : a \in \  A_i \}, \\
 A^{(M)}_i \ &= \ \{ Nn(q-1) + a : a \in A_i, q \in [M] \},
\end{split}
\end{align}
for $i \in [n]$.

It is easy to check that, using (\ref{parteq13a}) for the latter,
that for $i,j \in [n]$
\begin{align}\label{parteq13b}
\begin{split}
Q(A^*_i,A^*_j) \  &= \ - Q(A_i,A_j), \\
Q(A^{(M)}_i,A^{(M)}_j) \  &= \ M \cdot Q(A_i,A_j).
\end{split}
\end{align}

So we see that
\begin{equation}\label{parteq10}
R[\A^*] = R[\A]^{-1}, \qquad
R[\A^{(M)}] = R[\A].
\end{equation}

Furthermore, for $i \in [n]$
\begin{align}\label{inteq17partaaa}
\begin{split}
\s(A^*_i) \ &= \ (Nn + 1)N - \s(A_i), \\
\s(A^{(M)}_i) \ &= \ \frac{Nn\cdot M(M-1) \cdot N}{2} \ + \ M \cdot \s(A_i).
\end{split}
\end{align}
It follows that if $\A$ is proper, then $\A^*$ and $\A^{(M)}$ are proper.

 \vspace{.5cm}

 \begin{theo}\label{partcor01a} If $\A = \{ A_1, \dots, A_n \}$ is an $n$ partition of $[Nn]$, then
 there exists $\A'' =  \{ A_1'', \dots, A_n'' \}$
  an $n$ partition of $[(N+2)n]$ with $R[\A''] = R[\A]$. \end{theo}

  \begin{proof} Define $C_i = \{Nn + i, Nn + 2n - i +1 \}$ for $i \in [n]$. If $j > i \in [n]$, then $C_j \hookrightarrow C_i$. So
  (\ref{parteq14}) implies $Q(C_i,C_j) = 0$ for distinct $i, j \in [n]$.
  Let $A_i'' = A_i \cup C_i$. Since $A_i < C_j$ for all $i,j$  (\ref{parteq13a}) implies
  that $Q(A_i'',A_j'') = Q(A_i,A_j) + Q(C_i,C_j) = Q(A_i,A_j)$ for distinct $i, j \in [n]$.

  \end{proof}
 \vspace{.5cm}

 \begin{df}\label{partdef00aaa} For $\A = \{ A_1, \dots, A_n \}$  an $n$ partition of $[Nn]$ and
 $\B = \{ B_1, \dots, B_m \}$  an  $m$ partition of $[Nm]$,  the \emph{domination product}\index{domination product}
 $\B \rhd \A = \{ A_1, \dots, A_{m+n} \}$
 is the $m+n$ partition of $[N(m+n)]$ which extends $\A$ by
  \begin{equation}\label{parteq10aaa}
  A_{n+p} \ = \ \{ Nn + j : j \in B_p \} \quad \text{for} \ p \in [m].
  \end{equation}
  \end{df}
 \vspace{.5cm}

 This is the partition version of the tournament construction given in (\ref{gameeq01}). To be precise:

 \begin{theo}\label{parttheo00bbb} For $\A = \{ A_1, \dots, A_n \}$  an $n$ partition of $[Nn]$ and
 $\B = \{ B_1, \dots, B_m \}$  an $m$ partition of $[Nm]$, let $R[\B]'$ be the tournament on
 $[n+1,n+m] = \{ n+1, \dots, n+m \}$ so that $n+i \to n+j$ in $R[\B]'$ if and only if $i \to j$ in $R[\B]$.
 Thus, $R[\A]$ and $R[\B]'$ are tournaments on disjoint sets. Furthermore,
 $ R[\B \rhd \A]  =  R[\B]' \rhd R[\A]$.
  \end{theo}

  \begin{proof} For $p,q \in [m]$, clearly, $Q(A_{n+p},A_{n+q}) = Q(B_p,B_q)$ and for $p \in [m], q \in [n]$
  $Q(A_{n+p},A_q) = N^2$ since $A_{n+p} > A_q$.

  \end{proof}
 \vspace{.5cm}

\begin{df}\label{partdef00} For $A \subset [M]$ and  $B \subset \N$, not necessarily disjoint, the
\emph{lexicographic product}\index{lexicographic product} $B \ltimes A $ is defined by
\begin{equation}\label{parteq13c}
B \ltimes A \ = \ \{ M(b - 1) + a : b \in B, a \in A \}.
\end{equation}
\end{df}
 \vspace{.5cm}

 The name is adopted because if $a_1, a_2 \in [M]$ then $|a_1 - a_2| < M$ and so for $b_1, b_2 \in \N$ it
follows that
\begin{equation}\label{parteq13d}
(M(b_1 - 1) + a_1) > (M(b_2 - 1) + a_2) \quad \Longleftrightarrow \quad
\begin{cases} \ b_1 > b_2 \ \text{or} \\ b_1 = b_2 \ \text{and} \ a_1 > a_2.\end{cases}
\end{equation}

In particular, we see that $|B \ltimes A| = |B| \cdot |A|$. It is easy to check that
\begin{equation}\label{parteq13ccc}
\s(B \ltimes A) \ = \ M \cdot |A| \cdot (\s(B) - |B|) + |B| \cdot \s(A).
\end{equation}

Notice that the definition requires that we specify $M$.
 \vspace{.5cm}

\begin{lem}\label{partlem01a}(a) If $A_1, A_{2} \subset [M]$, and
 $B_1, B_2$ are disjoint  subsets of $\N$, then $B_1 \ltimes A_1$ and $B_2 \ltimes A_2$ are disjoint with
\begin{equation}\label{parteq13e}
Q(B_1 \ltimes A_1,B_2 \ltimes A_2) \ = \ Q(B_1,B_2) \cdot |A_1|\cdot |A_2|.
\end{equation}

(b) If $A_1, A_{2}$ are  disjoint subsets of $[M]$, and
 $B$ is a subset of $\N$, then $B \ltimes A_1$ and $B \ltimes A_2$ are disjoint with
 \begin{equation}\label{parteq13f}
Q(B \ltimes A_1,B \ltimes A_2) \ = \ Q(A_1,A_2) \cdot |B|.
\end{equation}
\end{lem}

\begin{proof} (a) Since $B_1$ and $B_2$ are disjoint, $(M(b_1 - 1) + a_1) > (M(b_2 - 1) + a_2)$ if and only if $b_1 > b_2$.
Hence, (\ref{parteq13e}) easily follows from (\ref{parteq11a}).

 (b) Write $B \ltimes A_p = \bigcup \{ \{ b \} \ltimes A_p : b \in B \}$ for $p = 1,2$.
 From (\ref{parteq13a}) it follows that
   \begin{equation}\label{parteq13h}
   Q(B \ltimes A_1,B \ltimes A_2) \ = \ \sum_{b \in B} Q(\{ b \} \ltimes A_1,\{ b \} \ltimes A_2).
   \end{equation}
   Clearly, $Q(\{ b \} \ltimes A_1,\{ b \} \ltimes A_2) = Q(A_1,A_2)$ for every $b \in B$ and so (\ref{parteq13f}) follows.

 \end{proof}
 \vspace{.5cm}

%
%

 \begin{df}\label{partdef00bbb} For $\A = \{ A_1, \dots, A_n \}$  an $n$ partition of $[Nn]$ and
$\B = \{ B_1, \dots, B_k \}$  a $k$-partition of $[Kk]$,  the \emph{lexicographic product}\index{lexicographic product}
$\B \ltimes \A = \{ C_1, \dots, C_{kn} \}$ is the  $kn$ partition of $[KNkn]$ defined by
 \begin{equation}\label{parteq20}
 C_{n(i-1) + j} \ = \ B_i \ltimes A_j  \quad \text{with} \ \ M = Nn.
 \end{equation}
 \end{df}

 From Lemma \ref{partlem01a} we obtain.

 \begin{lem}\label{partlem02} For $\B \ltimes \A = \{ C_1, \dots, C_{kn} \}$, we have
 \begin{align}
Q(C_{n(i_1-1) + j_1},C_{n(i_2-1) + j_2}) \ &= \ Q(B_{i_1},B_{i_2}) \cdot N^2 \quad \text{when} \ i_1 \not= i_2,\label{parteq21}\\
 Q(C_{n(i-1) + j_1},C_{n(i-1) + j_2}) \ &= \ Q(A_{j_1},A_{j_2}) \cdot K \quad \text{when} \ j_1 \not= j_2.\label{parteq22}
 \end{align}
 \end{lem}

%
%
%
\vspace{.5cm}

 From this computation we immediately obtain the following theorem. Note that the lexicographic product of proper partitions is proper by (\ref{parteq13ccc}).

 \begin{theo}\label{parttheo03} For $\A $  an $n$ partition of $[Nn]$ and
$\B $  a $k$ partition of $[Kk]$, the  $kn$ partition of $[NKnk]$ $\B \ltimes \A$ satisfies
$R[\B \ltimes \A] = R[\B] \ltimes R[\A]$.  If $\A$ and $\B$ are proper, then $\B \ltimes \A$ is proper. \end{theo}
  \vspace{.5cm}

 By adapting this construction we now show how to extend the digraph of a partition, by inserting an  additional vertex.

 Let $\A = \{ A_1, \dots, A_n \}$ be an $n$ partition of $[Nn]$ and  let $J \subset [n]$ with $J' = [n] \setminus J$. We will show how to obtain
 an $n+1$ partition whose tournament on $[n+1]$ restricts to $R[\A]$ on $[n]$ and with $n+1 \to j$ if and only if $j \in J$.

 First, we consider the extreme cases where $J = [n]$ or $J = \emptyset$, i.e. where $n+1$ is a maximum or minimum vertex.

 Let
 $\O_N = \{ [N] \}$ which we regard as a $1$ partition of $[N] = [N \cdot 1]$. Thus, $R[\O_N] = \emptyset[ [1] ]$, the trivial tournament on
 the singleton $[1]$.

 By Theorem \ref{parttheo00bbb} $\CC =  \O_N \rhd \A$ is an $n+1$ partition of $[N(n+1)]$ which extends $\A$ and which has
 $A_{n+1} \to A_p$ for all $p \in [n]$. This is the case with $J = [n]$.

 On the other hand, if $\CC' = \A \rhd \O_N = \{ C_1', \dots, C_{n+1}' \}$ and we permute, defining $C_p = C_{p+1}'$ for $p \in [n]$ and
 $C_{n+1} = C_1'$, then for $\CC = \{ C_1, \dots, C_n, C_{n+1} \}$, $R[\CC]$ restricts to $R[\A]$ on $[n]$ and
 $C_p \to C_{n+1}$ for all $p \in [n]$. This is the case with $J = \emptyset$.

  For the remaining construction we will use  $\B = \{B_1, B_2, B_3 \}$ with
   \begin{align}\label{parteq17a}
 \begin{split}
 B_1 \ &= \ \{  2, 4, 6 \}, \\
B_2 \ &= \ \{  2, 3, 8 \}, \\
B_3 \ &= \ \{  1, 5, 7 \},
 \end{split}
 \end{align}
 so that $B_1$ and $B_2$ are disjoint from $B_3$, but $B_1 \cap B_2 = \{ 2 \}$.
 Clearly, $B_1^+, B_3^+ \hookrightarrow B_2^+$ and $B_3^+ \twoheadrightarrow B_1^+$.

Using Lemma \ref{triplelem} or computing directly we see that
 \begin{equation}\label{parteq17aa}
 Q(B_2,B_3) \ = \ Q(B_3,B_1) \ = \ 2 \cdot 5 - 9 \ = \ 1.
 \end{equation}

  Now let $J \subset [n]$ with $J$ and $J' = [n] \setminus J$ nonempty.
  Define the  sets $ C_1,\dots,C_{n+1}  \subset [8Nn]$ using a variation of the lexicographic order construction.

  With $M = Nn$ define
     \begin{align}\label{parteq17b}
 \begin{split}
  C_i \ = \ & B_1 \ltimes A_i  \quad \text{for} \ i \in J, \\
  C_i \ = \ & B_2 \ltimes A_i \quad \text{for} \ i \in J', \\
   C_{n+1}& \ = \ B_3 \ltimes [N].
    \end{split}
 \end{align}
 \vspace{.5cm}

  \begin{lem}\label{partlem04} The sets $ C_1, \dots, C_{n+1} $ are pairwise disjoint with $|C_i| = 3N$ for
  $i \in [n+1]$ and  satisfying
 \begin{align}
Q(C_i,C_j) \ &= \ 3 \cdot Q(A_{i},A_j) \ \text{when} \  i,j \in J \ \text{or} \ i,j \in J' \label{parteq17c}\\
Q(C_i,C_j) \ &= \  Q(A_{i},A_j) \ \text{when} \  i \in J \ \text{and} \  j \in J'.\label{parteq17d} \\
Q(C_i,C_{n+1}) \ &= \ N^2 \ \text{when} \ i \in J', \label{parteq17e} \\
Q(C_{n+1},C_i) \ &= \ N^2 \ \text{when} \  i \in J.\label{parteq17f}
 \end{align}
 \end{lem}

  \begin{proof} Because the sets $B_1$ and $B_2$ are combined with different sets in $\A$, it is clear that
  $ C_1, \dots, C_{n+1} $ are pairwise disjoint.

  Equation (\ref{parteq17c}) follows from (\ref{parteq13f}).

  Equations (\ref{parteq17e}) and (\ref{parteq17f}) follow from (\ref{parteq13e}).

  For (\ref{parteq17d}) we write $B_0 = \{ 2 \}$, so that $B_1 = B_0 \cup B_1^+$ and
  $B_2 = B_0 \cup B_2^+$.  Since $B_1^+ \hookrightarrow B_2^+$ and $B_0 < B_1^+, B_2^+$ it follows that
        \begin{align}\label{parteq17g}
 \begin{split}
Q(B_1^+, B_2^+) \ &= \ 0,  \\
Q(B_1^+, B_0) \ = \ Q(B_2^+, &B_0) \ = \ 2.
     \end{split}
 \end{align}

 So for distinct $i, j \in [n]$, (\ref{parteq13e}) implies
         \begin{align}\label{parteq17ga}
 \begin{split}
Q(B_1^+ \ltimes A_i, B_2^+ \ltimes A_j) \ &= \ 0,  \\
Q(B_1^+ \ltimes A_i , B_0 \ltimes A_j) \ = \ Q(B_2^+ \ltimes A_i, &B_0 \ltimes A_j) \ = \ 2N^2.
     \end{split}
 \end{align}
 From (\ref{parteq13f}) we have
 \begin{equation}\label{parteq17h}
 Q(B_0 \ltimes A_i , B_0 \ltimes A_j) \ = \ Q(A_i,A_j).
 \end{equation}
 Because $B_1 \ltimes A_i = (B_0 \ltimes A_i) \cup (B_1^+ \ltimes A_i) $ and
  $B_2 \ltimes A_j = (B_0 \ltimes A_j) \cup (B_2^+ \ltimes A_j)$, (\ref{parteq13}), (\ref{parteq17ga})
  and (\ref{parteq17h}) imply (\ref{parteq17d}).

  \end{proof}
 \vspace{.5cm}

 Now assume that $ \CC = \{ C_1, \dots, C_{k} \} $ is a collection of pairwise disjoint subsets of $\N$ with $|C_i| = K$ for $i \in [k]$.
 We obtain a $k$-partition of $[Kk]$ $ \ \D = \{ D_1, \dots, D_{k} \} $ by \emph{packing}\index{packing}  $\CC$ as follows.
 Let $m_1,\dots, m_{Kk}$ number in order the points of $\bigcup_{p=1}^k C_p$ so that $m_i > m_j$ if and only if $i > j$. Define
 \begin{equation}\label{parteq17i}
 D_p = \{ i : m_i \in C_p \} \quad \text{for} \ p \in [k].
 \end{equation}
 Since the ordering of the elements is preserved by the renumbering it follows that
 \begin{equation}\label{parteq17j}
 Q(D_p, D_q) \ = \ Q(C_p, C_q) \quad \text{for} \ p, q \in [k].
  \end{equation}
 \vspace{.5cm}

  \begin{theo}\label{parttheo05}  Let $\A = \{ A_1, \dots, A_n \}$ be an $n$ partition of $[Nn]$ with $R[\A]$ a tournament on $[n]$,
  e.g. with $N$ odd. Assume that $R$ is a tournament on $[n+1]$ with restriction $R|[n]$ equal to $R[\A]$. There exists
 $\D = \{ D_1, \dots, D_{n+1} \}$ an $n+1$ partition of $[3N(n+1)]$ with $R[\D] = R$. \end{theo}

  \begin{proof} {\bfseries Case 1}\ (Either $R(n+1) = [n]$ or $R^{-1}(n+1) = [n]$).

  If $R(n+1) = [n]$, then we begin with $\CC =  \O_N \rhd \A$  and let $\D = \CC^{(3)}$, see (\ref{inteq17part}).
  This is an $n+1$ partition of $[3N(n+1)]$ and by (\ref{parteq13b}) $R[\D] = R$.

  Similarly if $R^{-1}(n+1) = [n]$, then we begin with $\CC$ the renumbering given above of $\CC' = \A \rhd \O_N$ and let $\D = \CC^{(3)}$.
 \vspace{.25cm}

 {\bfseries Case 2}\ ($R(n+1) = J \subset [n]$ with $J$ and $J' = [n] \setminus J$ nonempty).

 We apply the construction of (\ref{parteq17b}) to the
 partition $\A$  yielding the disjoint sets $\{  C_1, \dots,  C_{n+1} \}$.
 From Lemma \ref{partlem04} we see that
 $i \to j$ in $R$ if and only if $Q(C_i,C_j) > 0$.

 We obtain $\D$ by packing $\CC = \{ C_1, \dots, C_{n+1} \}$ to obtain an $n+1$ partition of $[3N(n+1)]$.

  \end{proof}
 \vspace{.5cm}

From this follows our main result and, in particular,  Theorem \ref{maintheopart2}.

 \begin{theo}\label{partcor06} For $n \geq 2$, if $R$ is a tournament on $[n]$ and $N$ is any odd integer with $N \geq 3^{n-2}$ or an even
 integer with $N \geq 2 \cdot 3^{n-2}$,
 then there exists an $n$ partition $\A$ of $[Nn]$
 with $R[\A] = R$. \end{theo}

 \begin{proof} We use induction on $n$ to prove that $R$ can be modeled by $\A$ an $n$ partition of $[3^{n-2} \cdot n]$.

 With $n = 2$, $\{ \{ 2 \}, \{ 1 \} \}$ is a $2$ partition of $1 \cdot 2 = 3^{n-2} \cdot n$.
 Reversing the two elements we obtain the other tournament on $[2]$.

 For the inductive step, we apply Theorem \ref{parttheo05}.

 By using $\A^{(2)}$ we obtain an $n$ partition on $[2 \cdot 3^{n-2}n]$ which models $R$.

 If $N = 3^{n-2} + 2m$ or  $N = 2 \cdot 3^{n-2} + 2m$, then the result follows by induction on $m$, using Theorem \ref{partcor01a} for the inductive step.

   \end{proof}

 \vspace{.5cm}

 We illustrate Theorem \ref{parttheo05} by beginning with $\A = \{ A_1, A_2 \} = \{ \{ 2 \}, \{ 1 \} \}$ and using $J = \{ 1 \}$.
 Thus, $N = 1$ and $n = 2$.
\begin{align}\label{parteq17k}
 \begin{split}
C_1 \ &= \ B_1 \ltimes A_1 \ = \ \{  2 + 2, 6 + 2, 10 + 2 \} \ = \ \{  4, 8, 12 \}, \\
C_2 \ &= \ B_2 \ltimes A_2 \ = \ \{  2 + 1, 4 + 1, 14 + 1 \} \ = \ \{  3, 5, 15 \}, \\
C_3 \ &= \ B_3 \ltimes [1] \ = \ \{  0 + 1, 8 + 1, 12 + 1 \} \ = \ \{  1, 9, 13 \}.
 \end{split}
 \end{align}

 Packing we obtain
    \begin{align}\label{parteq17m}
 \begin{split}
D_1 \ &=  \ \{  3, 5, 7 \}, \\
D_2 \ &=  \ \{ 2, 4, 9 \}, \\
D_3 \ &=  \ \{  1, 6, 8 \}.
 \end{split}
 \end{align}
 This is  example (\ref{exeq01}) with which we began.

 On the other hand, the exponential growth in Corollary \ref{partcor06} provides what is probably only a crude upper bound.
 For example, it shows that tournaments on $[5]$ can be modeled using $5$ partitions on $[27 \cdot 5] = [135]$.  The  example
 (\ref{exeq03})
 models the Rock-Paper-Scissors-Lizard-Spock tournament as a $5$ partition of $[30]$ and we will see in the next section that we can do even better.

We can mimic for partitions the extension construction of (\ref{gameeq04})  by using the following list.

 Define $\B = \{B_1, B_2, B_3, B_4 \}$ with
   \begin{align}\label{parteq18a}
 \begin{split}
 B_0 \ &= \ \{ 10 \}, \\
 B_1 \ &= \ \{  2, 7, 10 \}, \\
B_2 \ &= \ \{  4, 5, 10 \}, \\
B_3 \ &= \ \{  3, 8, 9 \}, \\
B_4 \ &= \ \{  1, 6, 11 \}.
 \end{split}
 \end{align}
Clearly: $B^-_2 \hookrightarrow B^-_1, B^-_3, B^-_4$, $B^-_3 \twoheadrightarrow B^-_1, B^-_4$ and $B^-_1 \twoheadrightarrow  B^-_4$.
From Lemma \ref{triplelem} we have

   \begin{align}\label{parteq18aa}
    \begin{split}
 Q(B_2,B_3)  \ = \  Q(B_3,&B_4) \  = \ Q(B_4,B_2) \ = \\ Q(B_3,B_1) \  =  \ Q(&B_1,B_4)  \  = \ 1, \\
 Q(B_1^-,B_2^-)  \ &= \ 0, \\
  Q(B_0, B_1^-) \ = \   &Q(B_0, B_2^-) \  = \ 2.
 \end{split}
 \end{align}

 Let $\A = \{ A_1, \dots, A_n \}$ be an $n$ partition of $[Nn]$ and let $J \subset [n]$ with $J' = [n] \setminus J$ with neither empty.
With $M = Nn$ define
     \begin{align}\label{parteq18bbb}
 \begin{split}
  C_i \ = \ & B_1 \ltimes A_i  \quad \text{for} \ i \in J',\\
  C_j \ = \ & B_2 \ltimes A_j \quad \text{for} \ j \in J, \\
 \bar U & \ = \ B_3 \ltimes [N], \\
 \bar V  & \ = \ B_4 \ltimes [N].
    \end{split}
 \end{align}

  \begin{lem}\label{partlem07aa} The sets $ C_1, \dots, C_{n}, \bar U, \bar V $ are pairwise disjoint with $|\bar U| = |\bar V| = |C_i| = 3N$ for
  $i \in [n]$ and they satisfy
 \begin{align}
Q(C_i,C_j) \ &= \ 3 \cdot Q(A_{i},A_j) \ \text{when} \  i,j \in J \ \text{or} \  i,j \in J' \label{parteq18c}\\
Q(C_i,C_j) \ &= \  Q(A_{i},A_j) \ \text{when} \ j \in J,  \ i \in J'.\label{parteq18d} \\
Q(C_j,\bar U) \ &= \ Q(\bar U,\bar V)\ = \  Q(\bar U,C_i)\ = \ N^2 \ \text{when} \ j \in J, \  i \in J'. \label{parteq18e} \\
Q(\bar V,C_j) \ &=  \ Q(C_i,\bar V)\ = \ N^2 \ \text{when} \ j \in J, \  i \in J'.\label{parteq18f}
 \end{align}
 \end{lem}

 \begin{proof} Using Lemmas \ref{partlem01} and \ref{partlem01a} with (\ref{parteq18aa}), the proof proceeds just as in Lemma \ref{partlem04}.

 \end{proof} \vspace{.5cm}

  \begin{theo}\label{parttheo08aa}  Let $\A = \{ A_1, \dots, A_n \}$ be an $n$ partition of $[Nn]$ with $R[\A]$ a tournament on $[n]$,
  and let $J \subset [n]$ and $J' = [n] \setminus J$ with neither empty. There exists
 $\D = \{ D_1, \dots, D_{n}, U, V \}$ an $n+2$ partition of $[3N(n+2)]$ with $R[\D] = R^+$ the extension of $R$ via $J$ and $U \to V$. \end{theo}

\begin{proof} Use $\CC = \{ C_1, \dots, C_n , \bar U, \bar V \}$ from Lemma \ref{partlem07aa} and then pack to obtain $\D$.

\end{proof}
\vspace{1cm}

 \section{Examples}\label{sec4}
\vspace{.5cm}

We saw at the end of Section \ref{sec1} that for any $N$ there exist, for $n$ sufficiently large, tournaments $R$ on $[n]$ which cannot be
modeled using an $n$ partition of $[Nn]$.

\begin{prop}\label{cycprop01} For $N$ a positive integer, let $R$ be a tournament on $[n]$ which cannot be modeled using an $n$ partition of $[Nn]$.
If $n$ is the minimum size of such an unobtainable tournament, then $R$ is a strong tournament. \end{prop}

\begin{proof} Assume that every tournament of size $k < n$ can be modeled using a $k$ partition of $[Nk]$. If $R$ is a tournament on $[n]$ which is
not strong, then by Proposition \ref{gameprop02}, it can be written as the domination product of two tournaments of smaller size. So by
Theorem \ref{parttheo00bbb} it can be modeled as the domination production of some $k_1$ partition of $[Nk_1]$ and a $k_2$ partition of
$[Nk_2]$ with $k_1, k_2$ positive integers such that $k_1 + k_2 = n$. Hence, $R$ can be modeled by  an $n$ partition of $[Nn]$.

\end{proof}
\vspace{.5cm}

In this section we will consider examples of tournaments on $[n]$ which can be modeled by using $n$ partitions of $[3n]$, i.e. with $N = 3$.

For $\pi$ a permutation of $[n]$ and $\A = \{ A_1, \dots, A_n \} $  an $n$ partition of $[Nn]$, define $\A^{\pi} = \{ A^{\pi}_1, \dots, A^{\pi}_n \} $ by
\begin{equation}\label{cyceq04}
A^{\pi}_i \ = \ A_{\pi^{-1}i},
\end{equation}
for $i \in [n]$, so that $Q(A^{\pi}_i,A^{\pi}_j) = Q(A_{\pi^{-1}i},A_{\pi^{-1}j})$ from which it follows that
$i \to j$ in $R[\A^{\pi}]$ if and only if $\pi^{-1}i \to \pi^{-1}j$ in $R[\A]$.  Thus,
\begin{equation}\label{cyceq05}
R[\A^{\pi}] \ = \ \pi R[\A].
\end{equation}

It follows that if $R[\A]$ is isomorphic to a tournament $R$ on $[n]$, then $R[A^{\pi}] = R$ for some permutation $\pi$ of $[n]$.

On the other hand, if  $\pi$ is a permutation of $[Nn]$ we can define $\pi(\A) = \{ \pi(A_1), \dots, \pi(A_n) \} $
with $\pi(A_p) = \{ \pi(j) : j \in A_p \}$, the image of $A_p$ by the map $\pi$. If we start with $\A$ an arbitrary $n$ partition of $[Nn]$,
%
then by varying the permutation $\pi$ we can
obtain any $n$ partition of $[Nn]$.

%
%
%
%
%
%
%
%

For $k = 1, \dots M-1$ call the transposition $(k, k+1)$ on $[M]$ a \emph{simple transposition}\index{simple transposition}.

\begin{lem}\label{cyclem04} With $M \geq 2$, every permutation on $[M]$ is a product of simple transpositions. \end{lem}

\begin{proof}  For a permutation $\pi$ we show that a product of simple transpositions applied to the sequence $\pi(1), \dots, \pi(M)$
transforms the sequence to the identity sequence $1, \dots, M$.

We first use induction on $k$ with $\pi(M) = M - k$, to obtain a sequence of simple transpositions after which the sequence terminates at $M$.

If $k = 0$, then no transpositions are necessary.

If $0 < k \leq M - 1$, apply the transposition $(M-k,M-k+1) = (M-k,M-(k-1))$ after which the sequence terminates at $M - (k-1)$. Now use the inductive
hypothesis to transform so that the sequence terminates at $M$.

For a permutation $\pi$ we show by induction on $M$ that a product of simple transpositions applied to the sequence $\pi(1), \dots, \pi(M)$
transforms the sequence to $1, \dots, M$.  This is trivial for $M = 2$.  Now assume $M > 2$.

By our first result we may assume that $\pi(M) = M$. Now
$\pi(1), \dots,$ $ \pi(M-1)$ is a permutation of $[M-1]$ and so by  inductive hypothesis, it can be transformed to the identity by
a sequence of simple transpositions on $[M-1]$.

\end{proof}
\vspace{.5cm}

Begin with an arbitrary   $n$ partition of $[Nn]$,  $\A = \{ A_1, \dots, A_n \} $ and $k \in [Nn - 1]$.  Assume that
$k \in A_{p_1}, k+1 \in A_{p_2}$ with $p_1 \not= p_2$. If $\pi$ is the simple transposition $(k,k+1)$, then
  \begin{align}\label{cyceq06}
    \begin{split}
    \pi(A_{p_1}) \ = \ &(A_{p_1} \setminus \{ k \}) \cup \{ k+1 \}, \\
    \pi(A_{p_2}) \ = \ &(A_{p_2} \setminus \{ k+1 \}) \cup \{ k \}, \\
    \pi(A_p) \ = \ &A_p \quad \text{for} \ p \not= p_1, p_2.
 \end{split}
 \end{align}

For $j \in [Nn]$ with $j \not= k, k+1$ it is clear that  $k > j$ if and only if $k+1 > j$. So it follows that
  \begin{align}\label{cyceq07}
    \begin{split}
    Q(\pi(A_p),\pi(A_q)) \ = \ &Q(A_p,A_q) \quad \text{if} \ \{ p, q \} \not= \{ p_1, p_2 \}, \\
    Q(\pi(A_{p_1}),\pi(A_{p_2})) \ = \ &Q(A_{p_1},A_{p_2}) + 2.
 \end{split}
 \end{align}

 Thus, either $R[\pi(\A)] = R[\A]$ or the only possible changes are
 \begin{itemize}
 \item[(i)] the reversal of the edge from
 $A_{p_2}$ to $A_{p_1}$, which occurs if \\ $Q(A_{p_2},A_{p_1}) = 1$,

 \item[(ii)] the elimination of the edge from
 $A_{p_2}$ to $A_{p_1}$, which occurs if $Q(A_{p_2},A_{p_1}) = 2$,

 \item[(iii)] the introduction of a new edge $A_{p_1}$ to $A_{p_2}$, which occurs if $Q(A_{p_2},A_{p_1}) = 0$.
 \end{itemize}
Notice that if $N$ is odd, then by (\ref{parteq11ccc}), cases (ii) and (iii) cannot occur.

 Call this operation a \emph{simple switch}\index{simple switch}.  From Lemma \ref{cyclem04}  it follows that we can get from any $n$ partition of $[Nn]$
 to any other by a sequence of simple switches.

 If $\A = \{ A_1, \dots, A_n \} $ is an $n$ partition of $[3n]$ we write $A_i = \{a^i_1 < a^i_2 < a^i_3 \}$ calling $a^i_s$ the \emph{level $s$
 element} for $s = 1, 2, 3$.  We call $\A$ a \emph{stratified partition}\index{partition!stratified} if
 \begin{equation}\label{cyceq08}
\{ a^i_s : i \in [n] \}    \ = \ \{ n(s-1) + j : j \in [n] \}, \qquad \text{for} \ s = 1, 2, 3.
\end{equation}
That is, the level one elements are $\{1, \dots, n \}$ and the level two elements are $\{n+1, \dots, 2n \}$ so that the
level three elements are $\{2n+1, \dots, 3n \}$.

 \begin{theo}\label{cyctheo05} If $\B = \{ B_1, \dots, B_n \} $ is an  $n$ partition of $[3n]$, then there
 exists $\A = \{ A_1, \dots, A_n \} $ a
 stratified $n$ partition of $[3n]$ with $R[\A] = R[\B]$. \end{theo}

 \begin{proof} As before label $B_i = \{ b^i_1 < b^i_2 < b^i_3 \}$. Notice that for $i, j \in [n]$
   \begin{align}\label{cyceq09}
    \begin{split}
  b^i_2  >  b^j_3 \ \text{or} \  b^i_1   >  b^j_3& \qquad \Longrightarrow \qquad Q(A_i,A_j)  \geq  2 \cdot 6 - 9 = 3, \\
    b^i_1 \ > \ b^j_2 \qquad &\Longrightarrow \qquad Q(A_i,A_j) \geq 2 \cdot 6 - 9 = 3.
 \end{split}
 \end{align}

 We first prove that by using a sequence of simple switches we can obtain $\{ b^i_3 : i \in [n] \}    \ = \ \{ 2n + j : j \in [n] \}$.
 If we let $m_3 = \min \{ b^i_3 : i \in [n] \}$ then this is equivalent to $m_3 = 2n+1$. Observe that
 always $m_3 \leq 2n+1 $ since $\{ b^i_3 : i \in [n] \} $ consists of $n$ distinct integers with maximum $3n$.
 We use induction assuming that $m_3 = 2n+1-r$. If $r = 0$ then there is nothing to prove.

 Assume that $r \geq 1$. Among the $n + r$ numbers in the interval $[2n+1-r,3n]$ some are level two elements and perhaps  some are even from level one.
 Let $k+1$ be the smallest such so that each of the numbers in the interval $[2n+1-r,k]$ is from level three.
 In particular, $k = b^j_3$ with $k+1 = b^i_s$ for $s = 2$ or $1$.
 If we do a simple switch of $k$ with $k+1$, to obtain $\bar B_i, \bar B_j$ then $k = \bar b^i_s$ and $k+1 = \bar b^j_3$.  Furthermore,
 (\ref{cyceq07})  implies that $Q(\bar B_i,\bar B_j) = Q(B_i,B_j) - 2$ which is still positive by (\ref{cyceq09}). As this is the only change in the
 $Q$ values it follows that $R[\bar \B] = R[\B]$. We continue down, switching until finally, we obtain $\B$ with $2n+1-r = b^i_p$ and with
 $m_3 = 2n+1-r+1 = 2n+1-(r-1)$. Now apply the induction hypothesis.

 Now assuming that $\B$ satisfies  $\{ b^i_3 : i \in [n] \}    \ = \ \{ 2n + j : j \in [n] \}$, we let
 $m_2 = \min \{ b^i_2 : i \in [n] \} = n+1-r$. If $r = 0$, then $\{ b^i_2 : i \in [n] \}    \ = \ \{ n + j : j \in [n] \}$ and so
 $\{ b^i_1 : i \in [n] \}    \ = \ \{ j : j \in [n] \}$.  As above, $m_2 \leq n+1$ and we are done if $r = 0$.

  Assume that $r \geq 1$. Among the $n + r$ numbers in the interval $[n+1-r,2n]$ some are from level one.  Let $k+1$ be the smallest
  such so that each of the numbers in $[n+1-r,k]$ is from level two. In particular, $k = b^j_2$ with $k+1 = b^i_1$. As before we do a sequence
  of switches to get $\bar B$ with $\bar b^i_1 = n+1-r$ and with $m_2 = n+1-r+1 = n+1-(r-1)$. Applying the induction hypothesis we arrive at $\A$.

  \end{proof}
  \vspace{.5cm}

  The main result of this section is the observation that for arbitrary $n$ there is a game of size $2n + 1$ which can be modeled
  using  a $2n+1$ partition of $[3(2n+1)]$, that is, with $N = 3$.

\begin{theo}\label{cyctheo07} Let $R$ be the group game on $\Z_{2n+1}$ with game subset $[n]$.  That is, for $p, q \in \Z_{2n+1}$ $p \to q$ if and only if
$q - p $ (mod $2n+1$) lies in $[n]$. There exists $\A = \{ A_0, A_1, \dots, A_{2n} \} $ a proper, stratified $2n+1$ partition of $[3(2n+1)]$
with $R[\A] = R$.  That is,
 \begin{equation}\label{cyceq12}
 A_p \to A_q \quad \Longleftrightarrow \quad q-p \in [n] \ \text{mod} \ 2n+1.
 \end{equation}
 \end{theo}

 \begin{proof} Notice that for convenience of the algebraic description we are labelling the elements of the
 partition by $\Z_{2n+1} = \{ 0, 1, \dots, 2n \}$ rather than
 by $[2n+1] = \{ 1, 2, \dots, 2n+1 \}$ .

Define $\A = \{ A_0, A_1, \dots, A_{2n} \} $  the stratified $2n+1$ partition of $[3(2n+1)]$ by (with $j = 1, \dots n$)
\begin{align}\label{cyceq11}
    \begin{split}
   A_0 \ = \ &\{ 2n+1, 3n+2, 4n+3 \}, \\
   A_j \ = \ &\{ 2n+1-j, 3n+2-j, 4n+3+2j  \}, \\
   A_n \ = \ &\{ n+1, 2n+2, 6n+3  \}, \\
   A_{n+1} \ = \ &\{ n, 4n+2, 4n+4  \}, \\
   A_{n+j} \ = \ &\{ n+1-j, 4n+3-j, 4n+2+2j \}, \\
   A_{2n} \ = \ &\{ 1, 3n+3, 6n+2  \}, \\
    \end{split}
    \end{align}

Note first that $\s(A_p) = 9n + 6$ for all $p$ and so the partition is proper.

 Observe that for $p < q \in [0,n]$
 \begin{equation}\label{cyceq13}
a^q_1 < a^p_1 < a^q_2 < a^p_2.
\end{equation}
Thus, $A^-_{p} \twoheadrightarrow A^-_{q}$.

Similarly, for $p < q \in [1,n]$
 \begin{equation}\label{cyceq14}
a^{n+q}_1 < a^{n+p}_1 < a^{n+q}_2 < a^{n+p}_2.
\end{equation}
So $A^-_{n+p} \twoheadrightarrow A^-_{n+q}$.

From (\ref{parteq14}) Case (ii) and (\ref{parteq13a}) it follows that
\begin{align}\label{cyceq15}
    \begin{split}
    Q(A_p,A_q) \ = \ &1 \quad \text{for} \ p < q \in [0,n], \\
     Q(A_{n+p},A_{n+q}) \ = \ &1 \quad \text{for} \ p < q \in [n].
        \end{split}
    \end{align}

    Next note that
     \begin{equation}\label{cyceq16}
a^{n+p}_1  <  a^{q}_1 <  a^q_2 <  a^{n+p}_2
\end{equation}
for $p \in [n], q \in [0,n]$. Thus, each $A^-_q \hookrightarrow A^-_{n+p}$.  Furthermore,
\begin{align}\label{cyceq17aaa}
    \begin{split}
    a^j_3 > a^{n+j}_3 > a^{n+p}_3  \quad &\text{for} \ 1 \leq p \leq j, \\
    a^{n+j}_3 > a^{j-1}_3 > a^q_3  \quad &\text{for} \ 0 \leq q \leq j-1.
          \end{split}
    \end{align}
     From (\ref{parteq14}) Case (ii) and (\ref{parteq13a}) it follows that, for $j \in [n]$
\begin{align}\label{cyceq17}
    \begin{split}
      Q(A_j,A_{n+p}) \ = \ &1 \quad \text{for} \ 1 \leq p \leq j, \\
     Q(A_{n+j},A_{q}) \ = \ &1 \quad \text{for} \ 0 \leq q \leq j-1.
        \end{split}
    \end{align}

   The relations of (\ref{cyceq12}) follow from (\ref{cyceq15})  and (\ref{cyceq17}).

\end{proof}

In particular, with $n = 2$ we obtain the unique game of size $5$ via
\begin{align}\label{cyceq18}
    \begin{split}
   A_0 \ = \ &\{ 5, 8, 11 \}, \\
   A_1 \ = \ &\{ 4, 7, 13  \}, \\
   A_2\ = \ &\{ 3, 6, 15  \}, \\
   A_3 \ = \ &\{ 2, 10, 12  \}, \\
    A_4 \ = \ &\{ 1, 9, 14  \}, \\
    \end{split}
    \end{align}

    Even among proper, stratified partitions this representation is not unique.
    For example we can also obtain the game of size $5$ via
   \begin{align}\label{cyceq19}
    \begin{split}
   A_0 \ = \ &\{ 5, 7, 12 \}, \\
   A_1 \ = \ &\{ 4, 6, 14  \}, \\
   A_2\ = \ &\{ 3, 10, 11  \}, \\
   A_3 \ = \ &\{ 2, 9, 13  \}, \\
    A_4 \ = \ &\{ 1, 8, 15  \}, \\
    \end{split}
    \end{align}

    Using these and suitable simple switches one can show that every tournament on $[n]$ with $n \leq 5$ can be modeled using
    an $n$ partition of $[3n]$. By Proposition \ref{cycprop01} one need only consider strong tournaments with $2 < n \leq 5$.

There are three isomorphism classes of games of size seven, labeled Type I, II and III in Section 10 of \cite{A1}.
All three can be obtained using  $7$ partitions of $[21]$.
The Type I game
is given by Theorem \ref{cyctheo07} with $n = 3$:
\begin{align}\label{cyceq20}
    \begin{split}
   A_0 \ = \ &\{ 7, 11, 15 \}, \\
   A_1 \ = \ &\{ 6, 10, 17 \}, \\
   A_2 \ = \ &\{ 5, 9, 19  \}, \\
   A_3 \ = \ &\{ 4, 8, 21 \}, \\
   A_4 \ = \ &\{ 3, 14, 16 \}, \\
   A_5 \ = \ &\{ 2, 13, 18  \}, \\
    A_6 \ = \ &\{ 1, 12, 20  \}.
    \end{split}
    \end{align}

The game of Type II can be obtained via the following proper, stratified $7$ partition of $21$:
\begin{align}\label{cyceq21}
    \begin{split}
   A_0 \ = \ &\{ 7, 9, 17 \}, \\
   A_1 \ = \ &\{ 6, 12, 15 \}, \\
   A_2 \ = \ &\{ 5, 8, 20  \}, \\
   A_3 \ = \ &\{ 4, 11, 18 \}, \\
   A_4 \ = \ &\{ 3, 14, 16 \}, \\
   A_5 \ = \ &\{ 2, 10, 21  \}, \\
    A_6 \ = \ &\{ 1, 13, 19  \}.
    \end{split}
    \end{align}
%
%

    \begin{align}\label{cyceq22}
    \begin{split}
    A_0 \to A_1, A_2, A_4; \quad &A_1 \to A_2, A_3, A_5; \\
     A_2 \to A_3, A_4, A_6; \quad &A_3 \to A_0, A_4, A_5; \\
      A_4 \to A_1, A_5, A_6; \quad &A_5 \to A_0, A_2, A_6; \\
       A_6 \to A_0, &A_1, A_3.
        \end{split}
    \end{align}
    Each with a $Q$ value of $1$. This is group game on $\Z_{7}$ with game subset $\{ 1, 2, 4 \}$.

    By doing a $1,2$ simple switch we reverse the arrow $A_5 \to A_6$.
     By doing a $10,11$ simple switch we reverse the arrow $A_3 \to A_5$.
      By doing a $18,19$ simple switch we reverse the arrow $A_6 \to A_3$. Together these three simple switches reverse the
      $3$-cycle $\langle A_3, A_5, A_6 \rangle$. This yields a Type III game which is not a group game. The result is still
      a stratified,  proper partition.

 \vspace{1cm}

\bibliographystyle{amsplain}

\printindex

\end{document}